\documentclass[preprint,a4paper,authoryear]{imsart}
\usepackage{amssymb,amsthm,amsmath}
\usepackage[sort&compress]{natbib}
\usepackage{hyperref}

\numberwithin{equation}{section}

\newcommand{\R}{\mathbb{R}}

\newcommand{\Y}{\boldsymbol{Y}}

\newcommand{\Lb}{\boldsymbol{L}}

\theoremstyle{plain}
\newtheorem{theorem}{Theorem}[section]
\newtheorem{proposition}[theorem]{Proposition}

\theoremstyle{definition}

\theoremstyle{remark}

\newcounter{assumptionlevy}

\newcounter{assumptioneigen}

\begin{document}

\begin{frontmatter}

\title{Correction to: Multivariate CARMA processes, continuous-time state space models and complete regularity of the innovations of the sampled processes, Bernoulli 18, pp. 46-63, 2012}
\runtitle{}
\begin{aug}
	\author[A]{\fnms{Robert}~\snm{Stelzer}\ead[label=e1]{robert.stelzer@uni-ulm.de}\ead[label=u1,url]{https://www.uni-ulm.de/mawi/finmath}}

	\address[A]{Institute of Mathematical Finance, Ulm University, Helmholtzstrasse 18, D-89081 Ulm, Germany\printead[presep={,\ }]{e1,u1}}
	
\end{aug}

\begin{abstract}
A serious flaw in the proof of the equivalence of continuous time state space models and MCARMA processes spotted in \cite{FasenSchenk2024} is corrected. We point out that likewise an issue in the proof of Theorem 3.2 in \cite{brockwell2011parametric} can be resolved and, hence, any MCARMA process and linear state space model has both a controller and an observer canonical representation. Equivalently, the transfer function has both a left and right matrix fraction representation.
\end{abstract}



\end{frontmatter}
\section{Introduction}
Multivariate CARMA processes as introduced in \cite{marquardt2007multivariate} are used in various applications and have also been implemented in R packages (see e.g. \cite{Tomasson2018}).
The statistical inference theory developed in \cite{FasenKimmig2020,FasenKimmig2017,FasenMayer2022,FasenScholz2019,SchlemmStelzer2012EJS}, for instance, hinges crucially on the equivalence of the class of L\'evy-driven MCARMA processes to the class  of L\'evy-driven linear state space models, as identifiability is typically ensured by considering state space models in echolon form.

But as observed in \cite{FasenSchenk2024} the proof of this equivalence (Corollary 3.4) in \cite{schlemmmixing2010} is incorrect, since Appendix 2 of \cite{Caines1988} does not guarantee that the leading coefficient of the ``denominator'' in the left matrix fraction representation of the transfer function can be chosen to be the identity. The same problem - likewise spotted by \cite{FasenSchenk2024} -  with right matrix fractions arises in the proof of Theorem 3.2 in \cite{brockwell2011parametric}, where the authors refer to Lemma 6.3-8 of \cite{Kailath1980}, which also does not establish that the leading coefficient in the ``denominator'' can be taken to be the identity. 

Hence, we shall prove below that any L\'evy-driven linear state space model can be represented in both observer and controller canonical form and that its transfer function always has both a left and a right matrix fraction representation with the ``denominator'' having the identity as the leading coefficient. This in particular implies that Corollary 3.4 in \cite{schlemmmixing2010} and Theorem 3.2 in \cite{brockwell2011parametric} are correct.

In the following we refer the reader to our original paper \cite{schlemmmixing2010} for any unexplained notions and notation. We will refer to equations, definitions, theorems etc. in that paper by putting the letter ``P'' in front of the respective number.

\section{Results}
Note first that there is a typo in Equation (P3.4b). In line with  \cite[Theorem 3.12]{marquardt2007multivariate}  it should read
\[
\beta_{p-j} = I_{\{0,\ldots,q\}}(j)\left[-\sum_{i=1}^{p-j-1}{A_i\beta_{p-j-i}+B_{q-j}}\right].\qquad\qquad\qquad\qquad\hfill \mathrm{(P3.4b)}
\]
\begin{proposition}[Observer canonical form/MCARMA process representation]\label{prop:observer}
	For the output process $\mathbf{Y}$ of a L\'evy-driven state space model of the form (P3.5) there exist a $p\in\mathbb{N}$ and matrices $\mathcal{A,B,C}$ of the form (P3.4) such that $\mathbf{Y}$  is the output process of an MCARMA state space representation (linear state space model in observer canonical form) of the form (P3.3).
\end{proposition}
\begin{proof}
	Denote by $H(z)=C(z\mathbb{I}_N-A)^{-1} B$ the transfer function of the linear state space model (P3.5). Two linear state space models driven by the same L\'evy process produce the same output process $\mathbf{Y}$ if they have the same transfer function (see e.g. Lemma 3.2 of \cite{SchlemmStelzer2012EJS}). 
	
	Necessarily, every element of the matrix $H(z)$  is a rational function in $z$ with the degree of the denominator exceeding the degree of the numerator (see p. 106 in \cite{Brockett2015}). Now Proof 2 of Theorem 17.1 of \cite{Brockett2015} gives matrices $\mathcal{A,B,C}$ of the form (P3.4) with $H(z)=\mathcal{C}(z\mathbb{I}_N-\mathcal{A})^{-1} \mathcal{B}$. These matrices define a model of the  form (P3.3) with the given output process $\mathbf Y$.
	
\end{proof}
\begin{proposition}[Controller canonical form]\label{prop:controller}
	For the output process $\mathbf{Y}$ of a L\'evy-driven state space model of the form (P3.5) there exist a $p\in\mathbb{N}$ and matrices $\mathfrak{A,B,C}$ of the form 
	\begin{align}
	 \mathfrak{A} =& \left(\begin{array}{ccccc}
			0 & \mathbb{I}_m & 0 & \ldots & 0 \\
			0 & 0 & \mathbb{I}_m & \ddots & \vdots \\
			\vdots && \ddots & \ddots & 0\\
			0 & \ldots & \ldots & 0 & \mathbb{I}_m\\
			-\widetilde A_p & -\widetilde A_{p-1} & \ldots & \ldots & -\widetilde A_1
		\end{array}\right)\in M_{pm}(\R),\\
	\mathfrak{B}=&\left(0_m,\ldots,0_m,\mathbb{I}_m\right)^T\in M_{pm,m}(\R)\text{ and}\\
	\mathfrak{C}=&\left(\begin{array}{ccc}\widetilde B_0, & \ldots, & \widetilde B_{p-1}\end{array}\right)\in M_{d,pm}(\mathbb{R})
	\end{align}
such that $\mathbf{Y}$  is the output process of  a linear state space model in controller canonical form 
	\begin{equation}
		\mathrm{d}\widetilde{\boldsymbol{G}}(t)=\mathfrak{A}\widetilde{\boldsymbol{G}}(t)\mathrm{d}t+\mathfrak{B}\mathrm{d}\Lb(t),\quad \Y(t)=\mathfrak{C}\widetilde{\boldsymbol{G}}(t),\quad t\in\R.
	\end{equation}
	.
\end{proposition}
\begin{proof}
The proof is analogous to the one of Proposition \ref{prop:controller} only that one now uses Proof 1 of Theorem 17.1 of \cite{Brockett2015}.
\end{proof}
\begin{proposition}[Matrix fraction representations]
	Let $H$ be the transfer function  of a L\'evy-driven  state space model of the form (P3.5). There exist $p,q,\tilde q\in\mathbb{N}_0$, $p>q, p>\tilde q$ and
	polynomials $P\in M_d(\mathbb{R}[z]),Q\in M_{d,m}(\mathbb{R}[z])$ (left matrix fraction description) as well as $\widetilde P \in M_m(\mathbb{R}[z]), \widetilde Q\in M_{d,m}(\mathbb{R}[z])$ (right matrix fraction description) such that
	\[
	H(z)=C(z\mathbb{I}_N-A)^{-1} B=P^{-1}(z)Q(z)=\widetilde Q(z) \widetilde P^{-1}(z)\,\,\forall z \in \mathbb{C}.
	\]
	with
	\begin{eqnarray*}
		P(z)=\mathbb{I}_dz^p+A_1z^{p-1}+\ldots+A_p, &\qquad
		Q(z)=B_0z^q+B_1 z^{q-1}+\ldots+B_q,\\
	\widetilde	P(z)=\mathbb{I}_mz^p+\widetilde A_1z^{p-1}+\ldots+\widetilde A_p,&\qquad
\widetilde	Q(z)=\widetilde B_0z^{\tilde q}+\widetilde B_1 z^{\tilde q-1}+\ldots+\widetilde B_{\tilde q}.
	\end{eqnarray*}
\end{proposition}
\begin{proof}
	\emph{Left matrix fraction:}
Let  $\mathcal{A,B,C}$ be the matrices obtained by Proposition \ref{prop:observer}, define $P,p$ by the elements/dimensions of $\mathcal A$, set $q=p- \min(\{i=1,\ldots,p:\,\beta_i\not =0\}$ and $B_{q-j}=\beta_{p-j}+\sum_{i=1}^{p-j-1}{A_i\beta_{p-j-i}},\,j=0,\ldots,q$. Defining $Q$ accordingly the arguments in the first step of the proof of Theorem P3.3 combined with   \cite{marquardt2007multivariate},  Theorem 3.12, establish that $H(z)=P^{-1}(z)Q(z)$.

	\emph{Right matrix fraction:}
Likewise let  $\mathfrak{A,B,C}$  be the matrices obtained by Proposition \ref{prop:controller}  and define $\widetilde P, \widetilde Q, p, \tilde q$  via their elements/dimensions. Then the arguments in the proof of  \cite{brockwell2011parametric}, Theorem 3.2, show 	$H(z)=\widetilde Q(z) \widetilde P^{-1}(z)$.

That the degrees of the polynomials $P,\widetilde P$ agree follows from comparing Proof 1 and 2 of  Theorem 17.1 of \cite{Brockett2015}.
\end{proof}
Whereas in the original paper and in \cite{brockwell2011parametric}, respectively, the existence of the observer and controller canonical form, respectively, was incorrectly  a consequence of the existence of the matrix left and right factorization with the leading coefficient of the ``denominator'' being the identity, in our above flow of arguments it is now actually the other way round.

\bibliography{library}
\bibliographystyle{imsart-nameyear}

\end{document}